\documentclass[a4paper, 12pt]{amsart}

\usepackage[T1]{fontenc}
\usepackage[utf8]{inputenc}
\usepackage[english]{babel}

\usepackage{amssymb,amsmath,amsfonts,amsthm}

\usepackage{xcolor}
\usepackage[normalem]{ulem} 

\usepackage{tikz}
\usetikzlibrary{matrix,arrows,decorations.pathmorphing}
\usepackage{stackrel}
\usepackage{tikz-cd}

\usepackage[mathscr]{euscript} 

\usepackage{enumitem} 

\theoremstyle{plain} 
\newtheorem{thm}{Theorem}[section]
\newtheorem{prop}[thm]{Proposition}
\newtheorem{lmm}[thm]{Lemma}

\newtheorem{question}[thm]{Question}

\theoremstyle{definition} 
\newtheorem{dfn}[thm]{Definition}

\theoremstyle{remark} 
\newtheorem{rmk}[thm]{Remark}

\newcommand{\cC}{\mathcal{C}}

\newcommand{\cM}{\mathcal{M}}

\newcommand{\cT}{\mathcal{T}}
\newcommand{\cX}{\mathcal{X}}

\newcommand{\N}{\mathbb{N}}
\newcommand{\Z}{\mathbb{Z}}
\newcommand{\Q}{\mathbb{Q}}
\newcommand{\C}{\mathbb{C}}
\newcommand{\bbT}{\mathbb{T}}

\renewcommand{\epsilon}{\varepsilon}

\renewcommand{\phi}{\varphi}

\newcommand{\msout}{\bgroup\markoverwith{\textcolor{blue}{\rule[0.55ex]{2pt}{0.4pt}}}\ULon}
\newcommand{\mcross}{\bgroup\markoverwith{\textcolor{blue}{$\times$}}\ULon}

\newcommand{\bsout}{\bgroup\markoverwith{\textcolor{green}{\rule[0.55ex]{2pt}{0.4pt}}}\ULon} 

\newcommand{\bcross}{\bgroup\markoverwith{\textcolor{green}{$\times$}}\ULon} 

\newcommand{\fsout}{\bgroup\markoverwith{\textcolor{red}{\rule[0.55ex]{2pt}{0.4pt}}}\ULon} 

\newcommand{\fcross}{\bgroup\markoverwith{\textcolor{red}{$\times$}}\ULon} 

\begin{document}

\title[Moduli stacks of Riemann surfaces with symmetry]{Smooth
covers of moduli stacks of Riemann surfaces with symmetry}

\author{Fabio Perroni}
\thanks{}
\address[F.\ Perroni]{Dipartimento di Matematica e Geoscienze, Università degli Studi di Trieste, via Valerio 12/1, 34127 Trieste, Italy}
\email{fperroni@units.it}



\maketitle

\textit{Dedicated to  Fabrizio Catanese on the occasion of his 70th birthday}

\begin{abstract}
We construct explicitly a finite  cover of the moduli stack of compact Riemann surfaces with a given group of symmetries 
by a smooth quasi-projective variety.
\end{abstract}

\tableofcontents


\section{Introduction}
For a finite group $G$, the locus $M_g(G)$  (in the moduli space $M_g$) of curves
that have an effective action by $G$ plays an important role in the study of the geometry of $M_g$
(for example its singularities \cite{Cor87}, \cite{Cat12}), in  the study of
Shimura varieties (see e.g.  \cite{FPP}, \cite{MZ}, \cite{CFGP} and the references therein),
of totally geodesic subvarieties of $M_g$ \cite{EMMW}, and also in the classification of higher dimensional 
varieties (see e.g. \cite{Cat15}, \cite{FrGl}, \cite{Cat17}, \cite{LLR}). 
Similar loci in the moduli space of higher dimensional varieties have been studied in \cite{Li18}.

To investigate the geometry of $M_g(G)$ it is more natural 
to introduce  the moduli stack $\mathcal{M}_g(G)$ of genus $g$ compact Riemann 
surfaces with an effective action by $G$. Then we can obtain $M_g(G)$ as the image of a finite morphism $\mathcal{M}_g(G) \to M_g$.
In this paper we study some geometric properties of $\mathcal{M}_g(G)$.
If $g\geq 2$, $\mathcal{M}_g(G)$ is a complex orbifold, whose connected components
are in bijection with the set $\mathbb{T}$  of topological types of the $G$-actions.
For any $\tau \in \mathbb{T}$, let $\mathcal{M}_g(G, \tau)$ be the corresponding connected component.
It is  known that $\mathcal{M}_g(G, \tau)$ is isomorphic to  the  stack quotient
$\left[ \mathcal{T}_g^{\tau (G)} / {\rm C}_{\Gamma_g}(\tau (G)) \right]$,
where $\mathcal{T}_g^{\tau (G)}$ is the  locus  of points, in the Teichm\"uller space $\mathcal{T}_g$,
that are fixed by $\tau (G)$, and 
${\rm C}_{\Gamma_g}(\tau (G))$ is the centralizer of $\tau (G)$ in the mapping class group $\Gamma_g$.
Using the theory of level structures 
we define a smooth quasi-projective variety $Z_\tau$ such that 
$\mathcal{M}_g(G, \tau) \cong \left[ Z_\tau/ \bar{\rm C}_\tau\right]$, where $\bar{\rm C}_\tau$ is a finite group.
Hence $\mathcal{M}_g(G, \tau)$ is a smooth Deligne-Mumford stack and  
$Z_\tau \to \mathcal{M}_g(G, \tau)$ is a finite Galois cover. 
The disjoint union of the $Z_\tau$'s is a finite smooth cover of $\mathcal{M}_g(G)$, since
$\bbT$ is finite (see e.g. \cite{CLP16}). 
Notice that the existence of a smooth quasi-projective variety $Z$ and a finite flat morphism 
$Z\to \mathcal{M}_g(G)$ follows also from  \cite{KV04}.

\subsubsection*{Acknowledgements} 
I would like to thank Fabrizio Catanese for introducing me to the subject of moduli spaces of curves 
with symmetries, and for teaching me a lot about this. 
I am grateful  to Paola Frediani and Alessandro Ghigi for discussions that 
motivated the present article. 

The research was partially supported by the national projects
PRIN 2015EYPTSB-PE1 ``Geometria delle variet\`a algebriche''
and 2017SSNZAW 005-PE1 ``Moduli Theory and Birational Classification'',
by the research group GNSAGA of INDAM and by FRA  of the University of Trieste.


\section{Moduli spaces of curves with symmetry} 
Throughout the article $G$ is a finite group and $g$ is an integer greater or equal than $2$. 
Let $\mathcal{M}_g(G)$ be the stack, in the complex analytic category, 
whose objects are pairs $(\pi \colon \mathcal C \to B, \alpha)$,
where $\pi \colon \mathcal C \to B$ is a family of compact Riemann surfaces of genus $g$ and 
$\alpha \colon G \times \mathcal C \to \mathcal C$
is an effective (holomorphic) action of $G$ on $\mathcal C$ such that, for any $a\in G$, 
$\pi \circ \alpha (a, \_)=\pi$. 
A morphism $(\Phi, \varphi)\colon (\pi \colon \mathcal C \to B, \alpha) \to (\pi' \colon \mathcal C' \to B', \alpha')$
is a Cartesian diagram 
\begin{center}
\begin{tikzcd}
\mathcal C \arrow[r, "\Phi"] \arrow[d, "\pi"] & \mathcal C' \arrow[d, "\pi' "] \\
B \arrow[r, "\varphi"] & B' 
\end{tikzcd}
\end{center}
such that, for any $a\in G$, $\Phi \circ \alpha (a, \_) \circ \Phi^{-1}= \alpha'(a, \_)$. 

Using the Teichm\"uller space $\cT_g$, we are going to define a complex orbifold structure on $\mathcal{M}_g(G)$.
Given a compact, connected, oriented topological surface of genus $g$, $\Sigma_g$, recall that
a Teichm\"uller structure on a  Riemann surface $C$   is  
the isotopy class of an orientation preserving homeomorphism $f\colon C \to \Sigma_g$,
it will be denoted with $[f]$. Two  Riemann surfaces with Teichm\"uller structures 
$(C, [f]), (C',[f'])$ are isomorphic, if there exists an isomorphism $F \colon C \to C'$ such that
$[f]= [f' \circ F]$. We will denote with $[C, [f]]$ the class of $(C,[f])$.
Then, $\cT_g$ is the set of isomorphism classes $[C,[f]]$ of compact Riemann 
surfaces of genus $g$ with Teichm\"uller structures. 
The mapping class group of $\Sigma_g$, denoted by $\Gamma_g$,  is the group of all isotopy classes 
of orientation preserving homeomorphisms of $\Sigma_g$. 
There is a natural action of $\Gamma_g$ on $\cT_g$, given by
$$
[\gamma] \cdot [C,[f]] = [C, [\gamma \circ f]] \, , \quad \forall [\gamma] \in \Gamma_g , \, [C, [f]] \in \cT_g \, . 
$$
Furthermore, for any $[C,[f]] \in \cT_g$, the homomorphism
\begin{equation*}
\sigma_{[f]} \colon {\rm Aut}(C) \to \Gamma_g \, , \quad \Phi \mapsto [f\circ \Phi \circ f^{-1}]
\end{equation*} 
is injective and its  image is the stabilizer of $[C,[f]]$ in $\Gamma_g$, that we denote by ${\rm Stab}_{\Gamma_g}([C,[f]])$.
We collect in the following theorem several results about the Teichm\"uller space, for a proof and for more 
details we refer to \cite{AC}, \cite{ACG}.
\begin{thm}
$\mathcal{T}_g$ has a natural structure of a complex manifold which is 
homeomorphic to the unit ball in $\C^{3g - 3}$. The action of $\Gamma_g$ on $\cT_g$
is holomorphic and properly discontinuous. The map 
$$
\cT_g \to M_g \, , \quad [C, [f]] \mapsto [C] \, , 
$$
to the coarse moduli space of compact Riemann surfaces of genus $g$ yields  an isomorphism  
$\cT_g/\Gamma_g \cong M_g$.

Furthermore there is a universal family of Riemann surfaces of genus $g$ with Teichm\"uller structure 
$$
\eta \colon \cX_g \to \cT_g \, . 
$$
\end{thm}

Let  $\alpha$ be an effective action of $G$ on $C$, viewed as an injective group homomorphism
$\alpha \colon G \to {\rm Aut}(C)$. Let us choose an orientation preserving homeomorphism $f \colon C \to \Sigma_g$.
Then we have an injective homomorphism  
\begin{equation}\label{rho}
\rho := \sigma_{[f]} \circ \alpha \colon G \to \Gamma_g \, 
\end{equation}
such that $\rho (G) \subset {\rm Stab}_{\Gamma_g}([C, [f]])$.
Notice that, if $f' \colon C \to \Sigma_g$ is another orientation preserving homeomorphism, we obtain a different 
homomorphism $\rho' \colon G \to \Gamma_g$. However $[C,[f]]$ and $[C,[f']]$
belong to the same $\Gamma_g$-orbit, so there exists $[\gamma]\in \Gamma_g$
such that $[\gamma] \cdot [C, [f]] = [C, [f']]$ and $\rho' = [\gamma] \cdot \rho \cdot [\gamma]^{-1}$. 
This motivates the following definitions (that are already present in the literature,
 in slightly different forms, see e.g. \cite{CLP16} and the references therein).
\begin{dfn}\label{topological_type}
Let $\alpha \colon G \to {\rm Aut}(C)$ be an injective homomorphism. Let $\rho$ be defined  in \eqref{rho}.
The \textit{topological type} of the $G$-action $\alpha$ is the class of $\rho$
in ${\rm Hom}^{\rm inj} (G, \Gamma_g)/\Gamma_g$, where 
${\rm Hom}^{\rm inj} (G, \Gamma_g)$ is the set of injective group homomorphisms from $G$ to $\Gamma_g$,
${\rm Hom}^{\rm inj} (G, \Gamma_g)/\Gamma_g$ is the quotient under the action of  $\Gamma_g$
by conjugation. 
\end{dfn}

\begin{dfn}\label{GTg}
The \textit{Teichm\"uller space of  compact Riemann surfaces of genus $g$ with $G$-actions}
is the set
$$
\cT_g(G) = \{([C, [f]], \rho) \in \cT_g \times {\rm Hom}^{\rm inj} (G, \Gamma_g) \, | \, 
\rho (G) \subseteq {\rm Stab}_{\Gamma_g}([C, [f]])\} \, .
$$
\end{dfn}

Notice that, if ${\rm Hom}^{\rm inj} (G, \Gamma_g)\not= \emptyset$, then $\cT_g(G)\not= \emptyset$ by Nielsen realization problem \cite{K83}
and that it carries the following action  by $\Gamma_g$:
\begin{equation}\label{mcgact2}
[\gamma] \cdot ([C, [f]], \rho) = ([\gamma] \cdot [C, [ f]], [\gamma] \cdot \rho \cdot [\gamma]^{-1}) \, ,
\end{equation}
for $[\gamma]\in \Gamma_g$ and  $([C, [f]], \rho) \in \cT_g(G)$.

\begin{prop}\label{mainprop}
$\cT_g(G)$ is a complex manifold. Moreover there is an object  of $\cM_g(G)$,
$(\eta(G) \colon \cX_g(G) \to \cT_g(G), \alpha)$, such that 
the associated classifying morphism 
$\cT_g(G) \to \cM_g(G)$  induces an isomorphism 
$$
\left[ \cT_g(G) / \Gamma_g  \right] \cong \cM_g(G) \, ,
$$ 
where $\left[ \cT_g(G) / \Gamma_g \right]$ is the stack quotient associated to the action
\eqref{mcgact2}. In particular $\cM_g(G)$ has a structure of 
complex orbifold in the sense of \cite{ALR}, \cite[XII, §4.]{ACG}.
\end{prop}
\begin{proof}
It follows directly from Definition \ref{GTg} that 
$$
\cT_g (G) = \bigsqcup_{\rho \in {\rm Hom}^{\rm inj} (G, \Gamma_g)} \cT_g^{\rho (G)} \, ,
$$
where $\cT_g^{\rho (G)}$ is the  locus of points fixed by $\rho (G)$.
Since $\rho (G)$ is a finite group, $\cT_g^{\rho (G)}$ is a complex submanifold of $\cT_g$, so 
the first claim follows. 

For any $\rho \in {\rm Hom}^{\rm inj} (G, \Gamma_g)$, let $\eta (\rho) \colon \cX_g (\rho) \to \cT_g^{\rho (G)}$
be the restriction of the universal family. There is a natural effective action, $\alpha (\rho)$, of $G$ on $\cX_g (\rho)$
such that $(\eta (\rho)\colon \cX_g (\rho) \to \cT_g^{\rho (G)}, \alpha (\rho))$ is an object of $\cM_g(G)$.
Then we define $(\eta(G) \colon \cX_g(G) \to \cT_g(G), \alpha)$ as the disjoint union of these objects.

To prove the last statement recall that the objects of $\left[ \cT_g(G) / \Gamma_g  \right]$, over a base $B$,
are pairs  $(p\colon P\to B, f\colon P \to \cT_g(G))$, where  $p\colon P\to B$ is a principal
$\Gamma_g$-bundle and $f$ is a $\Gamma_g$-equivariant holomorphic map. 
Let $f^*(\eta (G))$ be the pull-back of the family $\eta (G)$. The action of $\Gamma_g$
on $P$ extends to a free  action on $f^*(\eta (G))$. So $f^*(\eta (G))$ descends to 
a family $\pi \colon \cC \to B$ over $B$. By construction there is a $G$-action, $\alpha$, on $\cC$,
in such a way that $(\pi \colon \cC \to B, \alpha)$ is an object of $\cM_g(G)$.
Similarly one associates to every arrow of $\left[ \cT_g(G) / \Gamma_g  \right]$  an arrow
of $\cM_g(G)$ obtaining an equivalence of categories. 
\end{proof}

Furthermore we have the following result.
\begin{thm}\label{mainthm}
$\cM_g(G)$ is a smooth algebraic Deligne-Mumford stack with quasi-projective coarse moduli space. 
\end{thm}
A proof of this theorem will be given in  Section \ref{covers},
where we define  a groupoid 
presentation of $\cM_g(G)$, $X_1 \stackrel [t]{s}{\rightrightarrows} X_0$, with
$X_0$ and $X_1$  smooth quasi-projective algebraic varieties, and such that 
$X_0 \to \cM_g(G)$ is finite \'etale. 

\subsection{On the homology of $M_g(G)$}\label{homology}
Let $\bbT \subset {\rm Hom}^{\rm inj} (G, \Gamma_g)$ be a set of representatives of 
topological types of $G$-actions (see Definition \ref{topological_type}).
Notice that, for any $\tau \in \bbT$, 
$$
{\rm Stab}_{\Gamma_g}(\tau) = {\rm C}_{\Gamma_g}(\tau (G)) \, ,
$$
where ${\rm C}_{\Gamma_g}(\tau (G))$ is the centralizer of $\tau (G)$ in $\Gamma_g$.
Therefore we have an homeomorphism
$$
\cT_g(G) / \Gamma_g \cong \bigsqcup_{\tau \in \bbT} \cT_g^{\tau (G)}/{\rm C}_{\Gamma_g}(\tau (G)) \, .
$$

\begin{dfn}\label{MgGtau}
Let $\tau \in \bbT$. 
The moduli space of compact Riemann surfaces of genus $g$ with $G$-action of topological type $\tau$
is defined as $\cT_g^{\tau (G)}/{\rm C}_{\Gamma_g}(\tau (G))$ and will be denoted by $M_g(G,\tau)$. 
\end{dfn}

We have the following theorem from \cite{Harvey}, \cite{Cat00} 
(see also  \cite{GoHa} and the references therein).
\begin{thm}\label{HarCat}
$\cT_g^{\tau(G)}$ is bi-holomorphic to $\cT_{g',d}$ where
$g'$ is the genus of $C/G$,  $[C,[f]] \in \cT_g^{\tau(G)}$, and $d$ is the number of branch points 
of the quotient map $C \to C/G$. In particular $\cT_g^{\tau(G)}$ is homeomorphic to the unit ball
in $\C^{3g'-3+d}$.
\end{thm}

Let us give an interpretation of the previous theorem.
Given a point $[C, [f]] \in \cT_g^{\tau(G)}$, let 
$\alpha \colon \sigma_{[f]}^{-1} \circ \tau  \colon G \to  {\rm Aut}(C)$.
Then $f\colon (C, \alpha) \to (\Sigma_g, f\circ \alpha \circ f^{-1})$ is $G$-equivariant, so it yields 
an orientation-preserving homeomorphism of  $\theta \colon C'/G \to \Sigma_g'/G$,
where $C'\subseteq C$ and $\Sigma_g' \subseteq \Sigma_g$
are the loci of points with trivial stabilizer. The bi-holomorphism in Theorem \ref{HarCat}
sends $[C,[f]]$ to  $\left[ \Sigma_g'/G, [\theta] \right] \in \cT_{g',d}$.

We describe now the map $\cT_{g', d} \to M_g(G,\tau)$, which is the composition of
the inverse of the previous one, $\cT_{g',d} \to \cT_g^{\tau(G)}$,  with  the quotient 
$\cT_g^{\tau(G)} \to \cT_g^{\tau (G)}/{\rm C}_{\Gamma_g}(\tau (G))$.
Let $p \colon \Sigma_g \to \Sigma_g/G$ be the quotient map, let $y\in \Sigma_g'/G$,
and let $x \in p^{-1}(y)$.
Let $\mu \colon \pi_1\left( \Sigma_g'/G, y \right) \to G$ be the  monodromy of  
the covering $p$ restricted to $\Sigma_g'$.
For any $[D', [\theta]] \in \cT_{g',d}$ the composition
$$
\mu \circ \theta_* \colon \pi_1(D', \theta^{-1}(y)) \to G \, 
$$
gives a $G$-cover $C' \to D'$, hence  a point 
$[C, \alpha] \in M_g(G,\tau)$. Then the map $\cT_{g', d} \to M_g(G,\tau)$ 
sends $[D', [\theta]]$ to $[C, \alpha]$.

Notice that the class of $\mu$,
$$
[\mu] \in {\rm Hom}(\pi_1\left( \Sigma_g'/G, y \right) , G)/G \, ,
$$ 
where $G$ acts by conjugation, does not depend from the choice of $y$, and that
$\Gamma_{g',d}$ acts on ${\rm Hom}(\pi_1\left( \Sigma_g'/G, y \right) , G)/G$.

It follows from this that
$$
M_g(G, \tau ) \cong \cT_{g',d}/{\rm Stab}_{\Gamma_{g',d}}([\mu]) \, .
$$
In particular the homology of $M_g(G,\tau)$ with rational coefficients
can be computed as the homology of the group ${\rm Stab}_{\Gamma_{g',d}}([\mu])$:
$$
H_n(M_g(G,\tau) ; \Q) \cong H_n({\rm Stab}_{\Gamma_{g',d}}([\mu]); \Q)  \, .
$$

Let now $(\alpha_1, \beta_1, \ldots , \alpha_{g'}, \beta_{g'}, \gamma_1, \ldots , \gamma_d)$
be  a geometric basis
of the fundamental group $\pi_1(\Sigma_{g'} \setminus \{ y_1, \ldots , y_d\}, y)$
(here we follow the notation of \cite{CLP16}). 
Setting $a_i = \mu (\alpha_i), b_i= \mu (\beta_i), c_j=\mu (\gamma_j)$, we obtain
an element
$$
[(a_1, b_1, \ldots , a_{g'}, b_{g'}, c_1, \ldots , c_d)] \in G^{2g'+d}/G \, ,
$$
where $G$ acts diagonally by conjugation. This yields an injective correspondence between 
classes $[\mu]$ of monodromies and elements of $G^{2g'+d}/G$.
$\Gamma_{g', d}$
acts on $G^{2g'+d}/G$ (see e.g. \cite[Sec 2.]{CLP16}) and 
the homology of ${\rm Stab}_{\Gamma_{g',d}}([\mu])$ is isomorphic to 
the equivariant homology of the orbit 
of $[(a_1, b_1, \ldots , a_{g'}, b_{g'}, c_1, \ldots , c_d)]$ 
under the action of $\Gamma_{g',d}$. 
This motivates  the following 
\begin{question}\cite{CLPHstab}
Let  $n, d \in \N$. Are there constants $a, b$ such that the dimension of 
$H_n^{\Gamma_{g',d}}(G^{2g'+d}/G ; \Q)$ is independent of $g'$, in the range $g' > an+b$?
\end{question}
We proved in \cite{CLP16} that the previous question has an affermative answer in the case where $n=0$.

\section{Level structures and Teichm\"uller structures}
In this section we recall some basic results  and we fix the notation,
for the proofs and for more details we refer to \cite[Chapter XVI]{ACG}.
Given two groups $G_1$ and $G_2$ we denote with ${\rm Hom}(G_1, G_2)$ the set of homomorphisms from $G_1$ to $G_2$.
Notice that there is an action of $G_2$ on ${\rm Hom}(G_1, G_2)$, which is induced by the action of $G_2$ on itself by conjugation
($(g,h)\mapsto g^{-1}hg$).
An \textbf{exterior group homomorphism} 
from $G_1$ to $G_2$ is an element of the quotient set ${\rm Hom}(G_1, G_2)/G_2$.
For any $\varphi \in {\rm Hom}(G_1, G_2)$ we denote with $\hat \varphi$
its class in ${\rm Hom}(G_1, G_2)/G_2$. 
We will say that $\hat \varphi$ is an exterior group monomorphism (respectively epimorphism, isomorphism)
if $\varphi$ is a monomorphism  (respectively epimorphism, isomorphism).

Let  $X$ be a path connected topological space. 
For any pair of points $x, y \in X$ and for any continuous path $\gamma$ from $x$ to $y$,
let $\varphi_{\gamma} \colon \pi_1(X,x) \to \pi_1(X,y)$ be the isomorphism that sends any $[c]\in \pi_1(X,x)$
to $[\gamma^{-1}\cdot c \cdot \gamma]$, where $c$ is a loop in $X$ based at $x$ and $[c]$ is its homotopy class.\\
Let $H$ be a  group. We will identify two exterior homomorphisms $\hat \alpha \in {\rm Hom}(\pi_1(X,x), H)/H$ and 
$\hat \beta  \in {\rm Hom}(\pi_1(X,y), H)/H$
if $\hat \alpha = \widehat{\beta \circ \varphi_{\gamma}}$. 
Notice that this definition does not depend on the choice of $\gamma$ and yields an equivalence relation on the disjoint union
$$
\bigsqcup_{x\in X} {\rm Hom}(\pi_1(X,x), H)/H \, .
$$
The class of $\hat \alpha$ will be denoted with $[\alpha]$.

\begin{dfn}
A \textbf{Teichm\"uller structure of level $H$ on  $X$} 
is the class $[\alpha]$ of an exterior epimorphism  $\hat \alpha \in {\rm Hom}(\pi_1(X,x), H)/H$, for  $x\in X$.
\end{dfn}

Let now $C$ and $C'$ be two compact Riemann surfaces. Let $[\alpha]$ and $[\alpha']$
be two Teichm\"uller structures of level $H$ on $C$ and $C'$, respectively. 
We  say that the pairs   $(C, [\alpha])$  and $(C', [\alpha'])$
are isomorphic if there exists an isomorphism $F\colon C \to C'$ such that $[\alpha]=[\alpha' \circ F_*]$,
where $F_* \colon \pi_1 (C, x) \to \pi_1 (C', F(x))$ is the homomorphism induced by $F$, 
$\alpha \colon \pi_1 (C, x) \to H$ and $\alpha' \colon \pi_1 (C', F(x)) \to H$ are representatives of $[\alpha]$
and $[\alpha']$, respectively.
The isomorphism class of $(C,[\alpha])$ will be denoted $[C; \alpha]$.
The set of isomorphism classes of  compact Riemann surfaces  of genus $g$ 
with Teichm\"uller structures of level  $H$ will be denoted ${}_HM_g$.
When $g\geq 2$ it is possible to define a structure of complex analytic space on ${}_HM_g$,  called the 
\textit{moduli space of genus $g$ curves with Teichm\"uller structure of level  $H$}.
We refer to \cite{ACG} for more details.

\subsection{Level structures and Teichm\"uller structures}\label{lsts}
Let  $[\psi]$ be a Teichm\"uller structure of level $H$ on $\Sigma_g$.
The following map defines a morphism of complex analytic spaces:
\[
t_{[\psi]} \colon  \mathcal{T}_g \to {}_HM_g \, , \quad  [C, [f]] \mapsto [C; \psi \circ f_*] \, ,
\]
where $f$ is a representative of $[f]$ and $\psi \colon \pi_1 (\Sigma_g, f(x)) \to H$ is a representative of $[\psi]$.

\begin{rmk}
As explained in \cite{ACG}, the mapping class group $\Gamma_g$ acts on the set of Teichm\"uller structures of level $H$ on $\Sigma_g$
as follows: $[\psi] \cdot [\gamma] = [\psi \circ \gamma_*]$. Moreover every connected component of  ${}_H M_g$
coincides with $t_{[\psi]}(\mathcal{T}_g)$, for some $[\psi]$, we denote $t_{[\psi]}(\mathcal{T}_g)$ by  $M_g[\psi]$. 
Hence we have the following decomposition, 
\[
{}_H M_g = \coprod_{[\psi] \, {\rm mod} \, \Gamma_g}  M_g[\psi] \, .
\]
\end{rmk}
Let $\Lambda_{[\psi]} : = \{ [\gamma] \in \Gamma_g \, | \, [\psi] \cdot [ \gamma] = [\psi] \}$.
Then 
$$
M_g[\psi] = \mathcal{T}_g / \Lambda_{[\psi]} \, .
$$
Furthermore, if $\Lambda_{[\psi]}$ is a normal subgroup of $\Gamma_g$ and $\Gamma_g[\psi]:= \Gamma_g/\Lambda_{[\psi]}$,
then 
$$
M_g = M_g[\psi] / \Gamma_g[\psi] \, .
$$

We report the following result from \cite{ACG}, where  a characteristic subgroup of 
$\pi_1(\Sigma_g, x)$ is a subgroup that is mapped to itself by every automorphism of $\pi_1(\Sigma_g, x)$. 
\begin{lmm}\label{LambdaNormal}
If $\ker (\psi)$ is a characteristic subgroup, then $\Lambda_{[\psi]}$ is a normal subgroup of $\Gamma_g$.
\end{lmm}
\begin{proof}
Under our hypotheses  any automorphism $\gamma$ of 
$\pi_1(\Sigma_g)$ induces an automorphism $\bar{\gamma}$ of $H$ such that $\psi \circ \gamma = \bar{\gamma} \circ \psi$.
This yields a group homomorphism ${\rm Out}^+(\pi_1(\Sigma_g)) \to {\rm Out}(H)$, $[\gamma] \mapsto [\bar{\gamma}]$,
whose kernel is $\Lambda_{[\psi]}$, under the identification of ${\rm Out}^+(\pi_1(\Sigma_g))$ with $\Gamma_g$.
\end{proof}

In the case where $\ker (\psi)$ is a characteristic subgroup we can describe the action of $\Gamma_g$ on $M_g[\psi]$ as follows
(\cite{ACG}).
\begin{lmm}
Let $\ker (\psi)$ be a characteristic subgroup of $\pi_1(\Sigma_g)$. Then $\Gamma_g$ acts on $M_g[\psi]$ as follows:
$$
[\gamma] \cdot [C; \alpha] = [C; \bar{\gamma} \circ \alpha] \, ,
$$
where $\bar{\gamma}$ is the automorphism of  $H$ defined in the proof of Lemma \ref{LambdaNormal}.
\end{lmm}

Let now $m\in \Z_{\geq 1}$ and let $\chi_m \colon \pi_1(\Sigma_g) \to H_1(\Sigma_g, \Z/m\Z)$
be the composition of the natural morphism $\pi_1(\Sigma_g)\to H_1(\Sigma_g, \Z)$
with the reduction modulo $m$. Then $H_1(\Sigma_g, \Z/m\Z)$ is a strongly characteristic quotient 
of $\pi_1(\Sigma_g)$ (i.e. there is only one subgroup $K$ of $\pi_1(\Sigma_g)$
such that $\pi_1(\Sigma_g)/K$ is isomorphic to $H_1(\Sigma_g, \Z/m\Z)$),
 in particular $\ker (\chi_m)$ is a characteristic subgroup of $\pi_1(\Sigma_g)$.
In this situation we use the following notation:
$$
M_g[m] := M_g[\chi_m] \, , \quad \Gamma_g[m] := \Gamma_g[\chi_m] \, , \quad \Lambda_{[m]}:= \Lambda_{[\chi_m]} \, ,
\quad t_{[m]}:= t_{[\chi_m]} .
$$
Furthermore we have that 
$$
\Gamma_g[m] \cong {\rm Sp}_{2g} (\Z / m\Z)
$$
and moreover $M_g[m]$ coincides with the set of isomorphism classes of pairs $(C,  \rho)$,
where $\rho \colon H_1(C, \Z/m\Z) \to (\Z/m\Z)^{2g}$ is a symplectic isomorphism, with respect to the 
intersection form on $H_1(C, \Z/m\Z)$ and the standard symplectic form on $(\Z/m\Z)^{2g}$.

\section{Smooth covers of $\mathcal{M}_g(G, \tau)$}\label{covers}
Let  $\tau \colon G \to \Gamma_g$ be an injective homomorphism,
we denote with ${\rm C}_\tau$ the centralizer of $\tau (G)$ in $\Gamma_g$, and with ${\rm N}_\tau$ 
the normalizer of $\tau (G)$ in $\Gamma_g$.
For any subgroup $H\leq \Gamma_g$, its image under the quotient morphism 
$\Gamma_g \to \Gamma_g[m]=\Gamma_g/\Lambda_{[m]}$ will be denoted by $\bar H$.

\begin{prop}\label{smoothalgcover}
Let $m\geq 3$ be an integer and  $Z_\tau:={\mathcal{T}_g^{\tau(G)}}/({\Lambda_{[m]} \cap {\rm C}_\tau})$.
Then $Z_\tau$ is a complex manifold and, for $M_g(G, \tau)$ defined in Definition \ref{MgGtau}, 
the quotient morphism $\mathcal{T}_g^{\tau(G)} \to M_g(G, \tau)$ induces a finite morphism
\begin{equation}\label{sc}
Z_\tau \to M_g(G, \tau) \, ,
\end{equation}
which gives an isomorphism between $M_g(G, \tau)$ and $Z_\tau/\bar {\rm C}_\tau$.
\end{prop}
\begin{proof}
The first claim follows from the fact that $\Lambda_{[m]}$ acts freely on $\mathcal{T}_g$, 
since $m\geq 3$. To see this, let $[\gamma] \in \Lambda_{[m]}$
and let $[C,[f]] \in \mathcal{T}_g$ such that $[\gamma] \cdot [C,[f]] = [C,[f]]$.
Then $[\gamma] \in {\rm Aut}(C)$. Let $\bar \gamma_*$ be the induced automorphism of $H_1(C,\mathbb{Z}/m\mathbb{Z})$.
Since $[\gamma] \in \Lambda_{[m]}$,  $\bar \gamma_* ={\rm Id}$, then $[\gamma] = {\rm Id}_C$
(\cite[Prop. (2.8), p. 512]{ACG}).

The last claim follows from the fact that $\Lambda_{[m]}$ is a normal subgroup of $\Gamma_g$ (Lemma \ref{LambdaNormal})
and  ${\bar{\rm C}_\tau} = \frac{{\rm C}_\tau}{{\rm C}_\tau \cap \Lambda_{[m]}}$, 
so 
$$
M_g(G, \tau) = \mathcal{T}_g^{\tau(G)}/{\rm C}_\tau = Z_\tau/{\bar{\rm C}_\tau} \, .
$$
\end{proof}

Now we prove  that $Z_\tau$ is a quasi-projective algebraic variety.
Although this fact can be deduced from \cite{GoHa} or \cite{Cat12},  we give here, for completeness,  
an elementary proof by showing that there is a finite morphism
$Z_\tau \to \left( M_g[m]\right)^{\overline{\tau ( G )}}$, where $\left( M_g[m]\right)^{\overline{\tau ( G )}}$
is the set of points of $M_g[m]$ fixed by the action of $\overline{\tau ( G )}$.
By a finite morphism of complex analytic spaces we mean a proper morphism with finite fibers.

\begin{prop}
Let $t_{[m]} \colon {\mathcal T}_g \to M_g[m]$ be the morphism defined in Section \ref{lsts}, $m\geq 3$.
Then the following statements hold true:
\begin{itemize}
\item[1)] $t_{[m]}({\mathcal T}_g^{\tau (G)}) = \left( M_g[m]\right)^{\overline{\tau ( G )}}$;
\item[2)] let  $t_{[m]|}$ be the restriction of $t_{[m]}$ to ${\mathcal T}_g^{\tau (G)}$, then the 
following   diagram is commutative:
\[
\begin{tikzcd}
{\mathcal T}_g^{\tau (G)} \arrow[d, "t_{[m]|}"] \arrow[r] & Z_\tau \arrow[d, "q"] \\
\left( M_g[m]\right)^{\overline{\tau ( G )}} & \arrow[l, "p"] 
Z_\tau / \left( \frac{\Lambda_{[m]} \cap {\rm N}_\tau}{\Lambda_{[m]} \cap {\rm C}_\tau} \right)
\end{tikzcd}
\]
where the horizontal arrow to the top is the quotient map by $\Lambda_{[m]}\cap {\rm C}_\tau$, 
$q$  is the quotient map by 
$\frac{\Lambda_{[m]} \cap {\rm N}_\tau}{\Lambda_{[m]} \cap {\rm C}_\tau}$
and $p$ sends any equivalence class of the 
$(\Lambda_{[m]} \cap {\rm N}_\tau )$-action  to the corresponding equivalence class
of the $\Lambda_{[m]}$-action 
(under the identification of  $Z_\tau / \left( \frac{\Lambda_{[m]} \cap {\rm N}_\tau}{\Lambda_{[m]} \cap {\rm C}_\tau} \right)$
with $\frac{\mathcal{T}_g^{\tau (G)}}{\Lambda_{[m]} \cap {\rm N}_\tau}$);
\item[3)] the morphism $p\circ q$ is finite, therefore  $Z_\tau$ has a  structure of complex quasi-projective 
algebraic variety and $p\circ q$ is algebraic.
\end{itemize}
\end{prop}
\begin{proof}
1) Since $\Lambda_{[m]}$ is normal in $\Gamma_g$ (Lemma \ref{LambdaNormal}), $t_{[m]}$ is $G$-equivariant, therefore 
$t_{[m]}(\mathcal{T}_g^{\tau (G)}) \subseteq \left( M_g[m]\right)^{\overline{\tau ( G )}}$.
On the other hand, for any given $[C, \chi_m \circ f_*]\in \left( M_g[m]\right)^{\overline{\tau ( G )}}$ and $[\gamma] \in \tau (G)$,
from the equality
$$
[C, \chi_m \circ f_*] = \overline{[\gamma]} \cdot [C, \chi_m \circ f_*] 
$$
it follows that $[C,[f]]$ and $[\gamma]\cdot [C,[f]]$ map to the same point of $M_g[m]$. So, 
if $[\gamma] \not\in {\rm Stab}_{\Gamma_g}([C,[f]])$, it would be an element of finite order of $\Lambda_{[m]}$,
but this contradicts the fact that 
$\Lambda_{[m]}$ acts freely on $\mathcal{T}_g$.

The statement in 2) follows from the fact that $M_g[m] =t_{[m]} (\mathcal{T}_g) = \mathcal{T}_g/\Lambda_{[m]}$ and from 1).

To prove 3), we show that $p\circ q$ is closed and has finite fibers. 
Let us first show that $q$ has finite fibers, equivalently that 
$\Lambda_{[m]} \cap {\rm C}_\tau$ has finite index in $\Lambda_{[m]} \cap{\rm N}_\tau$. 
Notice that, for $\tau (G) = \{ h_1, \ldots , h_{|\tau ( G )|} \}$,
 ${\rm C}_\tau$ is the stabilizer of $(h_1, \ldots , h_{|\tau ( G )|}) \in \prod_{h\in \tau(G)} \tau (G)$
with respect to the   action of ${\rm N}_\tau$ given by conjugation on each factor. 
So $[ {\rm N}_\tau \, : \,  {\rm C}_\tau]< \infty$, hence also 
$[\Lambda_{[m]}\cap {\rm N}_\tau \, : \, \Lambda_{[m]}\cap {\rm C}_\tau]<
\infty$. To see  that $p$ has finite fibers, notice that for any
$[C,[\chi_m \circ f_*]] \in \left( M_g[m]\right)^{\overline{\tau ( G )}}$
$$
p^{-1}([C,[\chi_m \circ f_*]]) = \frac{\left( \Lambda_{[m]} \cdot [C,[f]] \right) \cap 
\cT_g^{\tau (G)}}{\Lambda_{[m]} \cap {\rm N}_\tau} \, ,
$$
and, since $\Lambda_{[m]}$ acts freely on $\cT_g^{\tau ( G )}$
(see the proof of Proposition \ref{smoothalgcover}), $p^{-1}([C,[\chi_m \circ f_*]])$ is in bijection with 
\begin{equation}\label{p_fibers}
\frac{ \{\lambda \in \Lambda_{[m]} \, | \, 
\lambda^{-1}\tau ( G) \lambda \subseteq  {\rm Stab}_{\Gamma_g}([C,[f]]) \} }{\Lambda_{[m]} \cap {\rm N}_\tau} \, .
\end{equation}
The claim follows since the map from the quotient set in \eqref{p_fibers} to the set of subgroups of 
${\rm Stab}_{\Gamma_g}([C,[f]])$, induced  by $\lambda \mapsto \lambda^{-1}\tau ( G) \lambda$,
is injective. \\
The fact that $p\circ q$ is closed follows from the fact that, for any  closed subset $A \subseteq \mathcal{T}_g^{\tau (G)}$,
the union of all $[\gamma]\cdot A$, $[\gamma] \in \Gamma_g$, is closed in $\mathcal{T}_g$
(see e.g. \cite{GoHa}, proof of Theorem 1).
Finally, $Z_\tau$  and $p\circ q$ are algebraic  by the Generalized Riemann Existence Theorem of Grauert-Remmert
\cite{GR58} (see  \cite[Theorem 3.2, Appendix B]{Har77}), 
$Z_\tau$ is quasi-projective since $\left( M_g[m]\right)^{\overline{\tau ( G )}}$ is so. 
\end{proof}

\medskip

\noindent \textit{Proof of Thm. \ref{mainthm}}
From the proof of Proposition \ref{mainprop} we have that 
$$
\cM_g (G) = \bigsqcup_{\tau \in \bbT} [\cT_g^{\tau (G)}/{\rm C}_\tau] \, , 
$$
where we use the notation of Section \ref{homology}. Since 
${\Lambda_{[m]} \cap {\rm C}_\tau}$ acts freely on $\cT_g^{\tau (G)}$ the stack quotient 
$[\cT_g^{\tau (G)}/({\Lambda_{[m]} \cap {\rm C}_\tau})]$ is represented by the variety
$Z_\tau = \cT_g^{\tau (G)}/({\Lambda_{[m]} \cap {\rm C}_\tau})$, hence
$$
\cM_g (G) = \bigsqcup_{\tau \in \bbT} \left[ Z_\tau/{\bar{\rm C}}_\tau\right] \, .
$$
The claim follows from the fact that the stack quotient of the algebraic variety 
$Z_\tau$ (Proposition \ref{smoothalgcover})
by the finite group ${\bar{\rm C}}_\tau$ is a Deligne-Mumford stack (see e.g. \cite[(4.6.1)]{LM-B}). \qed

\bibliographystyle{amsalpha}
\bibliography{bib}

\end{document}